\documentclass[11pt]{amsart}

\usepackage[USenglish]{babel}
\usepackage[T1]{fontenc} 
\usepackage[utf8,latin1]{inputenc}

\usepackage{amsmath,amsthm,amssymb,amsfonts}
\usepackage{dsfont}
\usepackage{mathtools}

\usepackage{enumitem}
\usepackage{nicematrix}
\usepackage{lmodern}

\usepackage[bookmarks=true]{hyperref}
\usepackage{xcolor}
\hypersetup{
    colorlinks,
    linkcolor={red!50!black},
    citecolor={blue!50!black},
    urlcolor={blue!80!black},
}

\usepackage{todonotes}
\usepackage[bottom=3cm, top=3cm, left=3.5cm, right=3.5cm]{geometry}

\mathtoolsset{showonlyrefs}  

\newcommand{\R}{\mathbb{R}}
\newcommand{\hei}{\mathbb{H}}
\newcommand{\mean}{\mathcal{H}}
\newcommand{\D}{\mathcal{D}}
\newcommand{\hess}{\mathrm{Hess}_{\mathrm{H}}}
\newcommand{\sr}{sub-Riemannian }
\newcommand{\vol}{\mu}
\newcommand{\diverg}{\mathrm{div}_\vol}
\newcommand{\area}{d\sigma_\mathrm{H}}
\newcommand{\riemarea}{d\sigma_\mathrm{R}}

\newcommand{\surf}{\Sigma}
\newcommand{\Nperp}{T}

\theoremstyle{plain}
\newtheorem{thm}{Theorem}[section]

\newtheorem{lem}[thm]{Lemma}
\newtheorem*{lem*}{Lemma}

\theoremstyle{definition}
\newtheorem{defn}[thm]{Definition}

\theoremstyle{remark}
\newtheorem{rmk}[thm]{Remark}

\author[Tommaso Rossi]{Tommaso Rossi}
\address{Univ. Grenoble Alpes, CNRS, Institut Fourier, F-38000 Grenoble, France \& SISSA, Via Bonomea, 265, 34136 Trieste, Italy}
\email{\href{mailto:tommaso.rossi1@univ-grenoble-alpes.fr}{tommaso.rossi1@univ-grenoble-alpes.fr}}

\title[Integrability of the sub-Riemannian mean curvature at degenerate points]{Integrability of the sub-Riemannian mean curvature at degenerate characteristic points in the Heisenberg group}

\date{\today}

\begin{document}
 

\begin{abstract}
We address the problem of integrability of the sub-Riemannian mean curvature of an embedded hypersurface around isolated characteristic points. The main contribution of this note is the introduction of a concept of mildly degenerate characteristic point for a smooth surface of the Heisenberg group, in a neighborhood of which the sub-Riemannian mean curvature is integrable (with respect to the perimeter measure induced by the Euclidean structure). As a consequence we partially answer to a question posed by Danielli-Garofalo-Nhieu in \cite{DGN-Integrability}, proving that the mean curvature of a real-analytic surface with discrete characteristic set is locally integrable.
\end{abstract}

\maketitle

\section{Introduction and statements}
Let $M$ be a sub-Riemannian manifold, and $\surf\subset M$ be an embedded hypersurface. The horizontal mean curvature $\mathcal{H} : \surf \to \R$ is a geometrical invariant which arises naturally in different areas of geometric analysis. It appears in the theory of minimal surfaces \cite{MR2354992,Hladky-Pauls,MR3319952,Pauls,Malchiodi}, in the study of the heat content asymptotics in \sr manifolds \cite{TW-heat-cont-hei,rizzi2020heat}, and in Steiner-type formulas for the volume of tubes around hypersurfaces \cite{Steiner}.

Of particular relevance in all aforementioned applications is the (local) integrability of $\mathcal{H}$, either with respect to the horizontal perimeter measure or the Riemannian one (cf.\ Section \ref{sec:prel} for precise definitions). An important fact is that, even for smooth hypersurfaces $\surf$, the horizontal mean curvature blows-up at the so-called characteristic points, where the subspace of horizontal directions is tangent to $\surf$, making the local integrability problem a non-trivial one.

For what concerns the sub-Riemaniannian perimeter measure $\sigma_\mathrm{H}$, as remarked first in \cite{DGN-Integrability} for the Heisenberg group, the blow-up of $\mathcal{H}$ is compensated by the degeneration of $\sigma_\mathrm{H}$, and thus $\mathcal{H} \in L^1_{\mathrm{loc}}(\surf,\sigma_\mathrm{H})$.

The aforementioned compensation fails if one replaces the sub-Riemannian perimeter measure $\sigma_\mathrm{H}$ on $\surf$ with the Riemannian one $\sigma_\mathrm{R}$, and in general $\mathcal{H}\notin L^1_{\mathrm{loc}}(\surf,\sigma_\mathrm{R})$. In all known examples, however, either $\surf$ is not very smooth, or the set of characteristic points has positive dimension, cf.\ \cite{DGN-Integrability}. On the other hand, if $\surf$ is at least $C^2$ and the characteristic set is discrete, no counter-examples to local integrability are known. Furthermore, a thorough analysis of many specific cases led the authors in \cite{DGN-Integrability} to conjecture that, under these assumptions, $\mean\in L^1_\mathrm{loc}(\surf,\sigma_\mathrm{R})$.

\medskip

In this note we address the question of local integrability with respect to the Riemannian perimeter measure in the Heisenberg group and in the more general context of three-dimensional contact manifolds. The case of non-degenerate and isolated characteristic points is elementary, and we can state the following result (proved in a setting that includes the Heisenberg group).

\begin{thm}
\label{thm:intro1}
Let $M$ be a three-dimensional contact \sr manifold, equip\-ped with a smooth measure $\vol$ and let $\surf\subset M$ be a $C^2$ embedded surface. Assume that all characteristic points of $\surf$ are isolated and non-degenerate. Then,
\begin{equation}
\mean\in L^1_{\mathrm{loc}}\left(\surf,\sigma_\mathrm{R}\right),
\end{equation}
where $\sigma_\mathrm{R}$ denotes the Riemannian induced measure by $\vol$ on $\surf$. 
\end{thm}

The case of \emph{degenerate} characteristic points is less understood, even for the case of smooth surfaces in the Heisenberg group $\hei$ (to which we restrict to for the rest of this introduction). The main result of this note is the definition of a concept of \emph{mildly degenerate} characteristic point for surfaces in the Heisenberg group, for which we are able to prove local integrability. In the following, $\hei$ is equipped with the Lebesgue measure, which induces the Riemannian measure $\sigma_\mathrm{R}$ on a smooth embedded hypersurface $\surf$.

\begin{thm}\label{thm:intro2}
Let $\surf\subset\hei$ be a smooth embedded surface. Assume that the all characteristic points of $\surf$ are isolated and mildly degenerate. Then
\begin{equation}
\mean\in L^1_{\mathrm{loc}}\left(\surf,\sigma_\mathrm{R}\right),
\end{equation}
where $\sigma_\mathrm{R}$ denotes the Riemannian induced measure on $\surf$. 
\end{thm}

The concept of \emph{mild degeneration} is based on a finite-order condition along an intrinsic curve $\mathcal{C} \subset \surf$ emanating from degenerate characteristic points, which to our best knowledge does not appear in previous literature (cf.\ Definitions \ref{defn:crit_curve} and \ref{d:mildly}). In particular, if $\surf$ is \emph{real-analytic}, all degenerate characteristic points are mildly degenerate. As a consequence, we have the following corollary, which answers affirmatively to the conjecture in \cite{DGN-Integrability}, at least for real-analytic surfaces.

\begin{thm}
\label{thm:intro3}
Let $\surf\subset\hei$ be a real-analytic embedded surface. Assume that the all characteristic points of $\surf$ are isolated. Then,
\begin{equation}
\mean\in L^1_{\mathrm{loc}}\left(\surf,\sigma_\mathrm{R}\right),
\end{equation}
where $\sigma_\mathrm{R}$ denotes the Riemannian induced measure on $\surf$. 
\end{thm}

Our results include and unify several previous examples of local integrability of the horizontal mean curvature present in literature. In particular, one can compare our results with Propositions $3.1$-$3.4$ in \cite{DGN-Integrability}: Propositions $3.1, 3.2$ are included in Theorem \ref{thm:intro1}, while Proposition $3.3$ is covered by Theorem \ref{thm:intro3}.

\begin{rmk}
Theorems \ref{thm:intro1}, \ref{thm:intro2} and \ref{thm:intro3} are proved below in a slightly stronger form, cf.\ Theorems \ref{thm:int_nondeg}, \ref{thm:int_deg} and \ref{thm:realanal}, respectively. Namely, in each case, we prove the local integrability of $\tfrac{1}{W}$, where $W$ denotes the norm of horizontal projection of the Riemannian horizontal normal to $\surf$. For this stronger result, the mild degeneration assumption is sharp and cannot be improved, cf.\ Remark \ref{rmk:sharpness}. In particular, it implies the local integrability of the horizontal mean curvature, but also the local integrability of the intrinsic horizontal Gaussian curvature as defined in \cite{Balogh-Gauss}, yielding Gauss-Bonnet-type theorems for surfaces with isolated and mildly degenerate (or non-degenerate) characteristic points. We refer to \cite[Thm.\ 1.1]{Balogh-Gauss} for details.
\end{rmk}

\textbf{Acknowledgments.} This work was supported by the Grants ANR-15-CE40-0018, ANR-18-CE40-0012 of the ANR, and the Project VINCI 2019 ref.\ c2-1212. The author is grateful to Luca Rizzi for his numerous comments and suggestions that allowed to greatly improve the quality of the note.

\section{Preliminaries}
\label{sec:prel}
Let $M$ be a smooth, connected $m$-dimensional manifold. For our purposes, a \sr structure on $M$ is defined by a subbundle of the tangent bundle $\D\subset TM$, which we call \emph{distribution}, and a metric on it, namely a positive, symmetric $(0,2)$-tensor on $\D$, denoted by $g$. 

If the distribution has rank $k$, then, locally in an open set $U$, we may describe it via a local orthonormal frame, namely a family of $k$ vector fields such that
\begin{equation} 
\label{eqn:loc_frame}
\D_p=\mathrm{span}_p\{X_1,\ldots,X_k\}\subset T_pM, \qquad\forall\,p \in U.
\end{equation}
We assume that the distribution is \emph{bracket-generating}, cf.\ \cite{nostrolibro} for details.


\subsection*{Divergence and horizontal gradient} Let $\vol$ be a smooth measure on $M$, defined by a positive tensor density. The \emph{divergence} of a smooth vector field is defined by
\begin{equation}
\diverg(X)\vol=\mathcal{L}_X \vol, \qquad\forall\,X\in\Gamma(TM),
\end{equation}
where $\mathcal{L}_X$ denotes the Lie derivative in the direction of $X$. The \emph{horizontal gradient} of a function $f\in C^1(M)$, denoted by $\nabla f$, is defined as the horizontal vector field (i.e. tangent to the distribution at each point), such that
\begin{equation}
g(\nabla f,V)= Vf,\qquad\forall\,V\in\Gamma(\D),
\end{equation} 
where $V$ acts as a derivation. In terms of a local orthonormal frame as in \eqref{eqn:loc_frame}, one has
\begin{equation}
\nabla f=\sum_{i=1}^k(X_if) X_i,\qquad\forall\,f\in C^1(M).
\end{equation}

\subsection*{Characteristic points} Let $\surf\subset M$ be a $C^1$ embedded hypersurface. We say that $p\in\surf$ is a \emph{characteristic point} if
\begin{equation}
\D_p\subseteq T_p\surf.
\end{equation}
We denote by $C(\surf)$ the set of characteristic points. Notice that $C(\surf)\subset\surf$ is a closed set, and it has zero measure if $\surf$ is at least $C^2$. We refer to \cite{Balogh-size, MR343011} for fine results about the size of $C(\surf)$ under suitable assumptions on the regularity of $\surf$.

The hypersurface $\surf$ can be locally described as follows: at $p\in M$, there exists a neighborhood $U\subset M$ of $p$ and $u\in C^1(U)$ such that
\begin{equation}
\surf\cap U=\{u=0\}, \qquad du|_{\surf\cap U}\neq 0.
\end{equation}
When $\surf$ is locally given as the zero-locus of $u$, then $p\in C(\surf)$ if and only if
\begin{equation}
\label{eqn:char_cond}
X_iu(p)=0,\qquad\forall\,i=1,\ldots,k.
\end{equation}

\subsection*{Horizontal Hessian} We introduce the horizontal Hessian for classifying characteristic points (cf.\ also \cite{barilari2020stochastic}). Fix an affine connection $\tilde{\nabla}$ on the distribution $\D$. Then, the horizontal Hessian of $u\in C^2(M)$ is the $(0,2)$-tensor on $\D$, defined as
\begin{equation} 
\hess(u)(V,W)=g(\tilde{\nabla}_V(\nabla u),W), \qquad\forall\,V,W\in\Gamma(\D).
\end{equation} 

While in general the definition of horizontal Hessian depends on the choice of the connection, it is intrinsic at characteristic points.

\begin{lem}
\label{lem:invariance}
Let $M$ be a \sr manifold and let $\surf=\{u=0\}\subset M$, where $u\colon M\rightarrow\R$ is a $C^2$ submersion on $\surf$. If $p\in C(\surf)$, then $\hess(u)|_p$ does not depend on the choice of the connection and thus is a well-defined bilinear map on $\D_p$.
\end{lem}

\begin{proof}
Let $\{X_1,\ldots,X_k\}$ be a local orthonormal frame around $p$, then, by definition of horizontal Hessian, for $V,W\in \Gamma(\D)$, we have
\begin{align}
\hess(u)(V,W) &=\sum_{i,j=1}^k V^i W^j g\left(\tilde{\nabla}_{X_i}\nabla u,X_j\right) \\
& =\sum_{i,j=1}^k V^i W^j g\left(\tilde{\nabla}_{X_i}\left(\sum_{l=1}^k (X_\ell u)X_\ell\right),X_j\right)\\
							&=\sum_{i,j,\ell=1}^k V^i W^j g\left((X_iX_\ell u)X_\ell+(X_\ell u)\tilde{\nabla}_{X_i}X_\ell,X_j\right),
\end{align}
using the linearity and the Leibniz formula for the connection. Finally, since $p\in C(\surf)$, $X_\ell u(p)=0$ for any $\ell=1,\ldots,k$, therefore we conclude that
\begin{equation}
\label{eqn:Hessian}
\hess(u)(V,W)|_{p}=\sum_{i,j=1}^k V^i W^j X_iX_ju(p),
\end{equation}
and the right-hand side does not depend on the choice of the connection $\tilde\nabla$.
\end{proof}


\begin{defn}
Let $M$ be a \sr manifold and $\surf\subset M$ be a $C^2$ embedded hypersurface in $M$. We say that $p\in C(\surf)$ is a \emph{non-degenerate characteristic point} if 
\begin{equation}
\det\left(\hess(u)|_{p}\right)\neq 0,
\end{equation}
where $u\in C^2$ is a local defining function for $\surf$ in a neighborhood of $p$. Notice that this property does not depend on the choice of $u$.
\end{defn}

\subsection*{Horizontal mean curvature} 
The \emph{horizontal mean curvature} at $p\in\surf$ is defined as
\begin{equation}
\label{eqn:mean_defn}
\mean(p)=-\left.\diverg\left(\frac{\nabla u}{\|\nabla u\|}\right)\right\rvert_p,
\end{equation}
where $u\in C^2$ is a local defining function for $\surf$ in a neighborhood of $p$. Notice that the value of $\mean$ does not depend on the choice of $u$.

\subsection*{(Sub-)Riemannian induced measure} Let $\nu$ be the horizontal unit normal to $\surf$, then the \sr induced measure $\sigma_\mathrm{H}$ on $\surf$ is the positive smooth measure with density $|i_\nu\mu|$. If $u$ is a local defining function for $\surf$, $\nu$ is given by
\begin{equation}
\label{eqn:normal}
\nu=\frac{\nabla u}{\|\nabla u\|}.
\end{equation}
Analogously, to define the Riemannian induced measure $\sigma_\mathrm{R}$, consider any Riemannian extension of the \sr structure, then replace $\nu$ with the Riemannian unit normal, which is given by \eqref{eqn:normal} with the Riemannian gradient. Notice that $\sigma_\mathrm{R}$ coincides with the $n-1$ dimensional Hausdorff measure on $\surf$ induced by the Riemannian structure. Moreover, it depends on the choice of a Riemannian extension, but this choice does not play any role concerning the integrability of the horizontal mean curvature.

\subsection{A general estimate for horizontal mean curvature}
We provide here a general estimate for the horizontal mean curvature. Let $M$ be a \sr manifold and $\surf\subset M$ be a $C^2$ embedded hypersurface. Without loss of generality, assume that $u\in C^2(M)$ is a global defining function for $\surf$, that is $u\colon M\rightarrow \R$ is a submersion on $\surf$ and $\surf=\{u=0\}$. Having fixed a local orthonormal frame for the \sr structure at a point $p$, say $\{X_1,\ldots,X_k\}$, recall that 
\begin{equation}
\label{eqn:hor_grad}
\nabla u=\sum_{i=1}^k (X_iu) X_i,
\end{equation}
and its norm, which we denote by $W$, is given by
\begin{equation}
\label{eqn:norm_grad}
W^2=\|\nabla u\|^2=g(\nabla u,\nabla u)=\sum_{i=1}^k (X_iu)^2.
\end{equation}
Therefore, we can write the horizontal mean curvature explicitly in terms of $\nabla u$ and $W$
\begin{equation}
\mean=-\diverg\left(\frac{\nabla u}{W}\right)=-\frac{1}{W}\Delta u+\frac{1}{W^2}g(\nabla u,W).
\end{equation}
Using formula \eqref{eqn:hor_grad} and \eqref{eqn:norm_grad}, we obtain
\begin{equation}
\label{eqn:mean_curv}
\mean=-\frac{1}{W}\Delta u+\frac{1}{W^3}\sum_{i,j=1}^k (X_iu) (X_ju)(X_iX_ju),
\end{equation}
which gives the estimate
\begin{equation}
\label{eqn:mean_est}
\begin{split}
|\mean|&\leq\frac{1}{W}\left(\left\|\Delta u\right\|_{L^\infty(U)}+\sum_{i,j=1}^k \left\|X_iX_ju\right\|_{L^\infty(U)}\right)\leq\frac{C_0}{W},
\end{split}
\end{equation}
for a suitable constant $C_0>0$, where we have used the inequality: $|X_i u|\leq W$, for any $i=1,\ldots,k$. Here $\|\cdot\|_{L^\infty(U)}$ denotes the supremum norm and $U$ is a relatively compact neighborhood of $p$. 

We recover the well-known integrability result for the horizontal mean curvature with respect to the \sr perimeter measure (see also \cite[Prop.\ 3.5]{DGN-Integrability}).

\begin{lem}
Let $M$ be a \sr manifold, equipped with a smooth measure $\vol$ and let $\surf\subset M$ be a $C^2$ embedded hypersurface in $M$. Then
\begin{equation}
\mean\in L^1_{\mathrm{loc}}\left(\surf,\sigma_\mathrm{H}\right).
\end{equation}
Here $\sigma_\mathrm{H}$ denotes the \sr induced measure on $\surf$.
\end{lem}

\begin{proof}
Let $U$ be a neighborhood of $p$ and consider $u\in C^2$, a local defining function for $\surf$ on $U$. Then, denoting by $W=\|\nabla u\|$, we have $\sigma_\mathrm{H} \propto W\sigma_\mathrm{R}$, up to a smooth never-vanishing function. Therefore, using \eqref{eqn:mean_est}, we obtain
\begin{equation}
\left|\int_{\surf\cap U}\mean \area\right|\leq C\int_{\surf\cap U}|\mean|W\riemarea\leq C\int_{\surf\cap U}\frac{C_0}{W}W\riemarea<+\infty .\qedhere
\end{equation}
\end{proof}

\section{Integrability for non-degenerate characteristic points in 3D contact sub-Riemannian manifolds}
Let $M$ be a smooth manifold of dimension $3$. Let $\omega$ be a contact one-form, that is such that $\omega\wedge d\omega\neq 0$. Then, the contact distribution is 
\begin{equation}
\D_p=\ker(\omega_p)\subset T_pM, \qquad\forall\,p\in M.
\end{equation}
By the non-degeneracy assumption on $d\omega$, $\D$ is a subbundle of rank 2 and is bracket-generating. Any metric on $\D$ defines a \sr structure on $M$. We will refer to $M$ as \emph{contact \sr manifold}. Recall the following normal form for an orthonormal frame (see \cite{MR1648710,MR1722037}).

\begin{thm}
\label{thm:normal_form}
Let $M$ be a 3D contact sub-Riemannian manifold, with contact 1-form $\omega$, and $\{X_1,X_2\}$ be a local orthonormal frame for $\D=\ker(\omega)$. There exists a smooth coordinate system $(x,y,z)$ such that
\begin{equation}
\begin{split}
X_1&=\partial_x-\frac{y}{2}\partial_z+\beta y(y\partial_x-x\partial_y)+\gamma y\partial_z,\\
X_2&=\partial_y+\frac{x}{2}\partial_z-\beta x(y\partial_x-x\partial_y)+\gamma x\partial_z,
\end{split}
\end{equation}
where $\beta=\beta(x,y,z)$ and $\gamma=\gamma(x,y,z)$ are smooth functions satisfying
\begin{equation}
\beta(0,0,z)=\gamma(0,0,z)=\partial_x\gamma(0,0,z)=\partial_y\gamma(0,0,z)=0.
\end{equation}
\end{thm}

We now prove the first integrability result for isolated non-degenerate characteristic points on general contact manifolds.

\begin{thm}
\label{thm:int_nondeg}
Let $\surf\subset M$ be a $C^2$ embedded surface, let $p$ be an isolated non-degenerate characteristic point and let $u\in C^2$ be a local defining function for $\surf$ in a neighborhood of $p$. Denote with $W$ the norm of the horizontal gradient of $u$. Then
\begin{equation}
\label{eqn:norm_nondeg}
\frac{1}{W}\in L^1_{\mathrm{loc}}(\surf,\sigma_\mathrm{R}),
\end{equation}
where $\sigma_\mathrm{R}$ denotes the Riemannian induced measure on $\surf$. In particular
\begin{equation}
\mean\in L^1_{\mathrm{loc}}(\surf,\sigma_\mathrm{R}). 
\end{equation}
\end{thm}

\begin{proof}
Introducing the normal form given by Theorem \ref{thm:normal_form}, we may assume that the characteristic point is at the origin and that $\surf$ is locally a graph around the origin. Indeed, recall that $\surf\cap U=\{u=0\}$ and $du\neq 0$ on $\surf\cap U$. However, since $\textbf{0}\in C(\surf)$
\begin{equation}
\label{eqn:1ord_cond}
0=d_pu(X_1)=\partial_xu(\textbf{0})\qquad\text{and}\qquad 0=d_pu(X_2)=\partial_yu(\textbf{0}),
\end{equation}
therefore $\partial_zu(\textbf{0})\neq 0$, which implies that, up to restricting $U$, $\surf\cap U=\{z=g(x,y)\}$, for some $C^2$ function $g\colon\R^2\rightarrow\R$. Moreover, in this coordinates, relations \eqref{eqn:1ord_cond} give first-order conditions on $g$
\begin{equation}
g(0,0)=\partial_xg(0,0)=\partial_yg(0,0)=0.
\end{equation}
Recall that we want to discuss the finiteness of the following integral
\begin{equation}
\label{eqn:int_norm}
\int_{\surf\cap U}\frac{1}{W}\riemarea=\int_V\frac{1}{W(x,y,g(x,y))}f(x,y)dxdy,
\end{equation}
where $V$ is a neighborhood of $(0,0)$ and $f$ is the Riemannian density in coordinates. Since the characteristic point is non-degenerate, up to restricting $V$, the map
\begin{equation}
\label{eqn:coord_change}
\varphi\colon
\begin{pmatrix}
x\\y
\end{pmatrix}
\mapsto
\begin{pmatrix}
\tilde x\\\tilde y
\end{pmatrix}
\quad\text{such that}\quad
\begin{cases}
\tilde x=X_1u(x,y,g(x,y))\\
\tilde y=X_2u(x,y,g(x,y))
\end{cases}
\end{equation} 
defines a smooth change coordinate on $V$. Indeed, its Jacobian at $(0,0)$ equals the determinant of the horizontal Hessian at the origin
\begin{equation}
\left.\det\left(\mathcal{J}\varphi\right)|_{(0,0)}=\det\begin{pmatrix}
\partial_xX_1u+g_x\partial_zX_1u & \partial_yX_1u+g_y\partial_zX_1u\\
\partial_xX_2u+g_x\partial_zX_2u & \partial_yX_2u+g_y\partial_zX_2u
\end{pmatrix}\right\rvert_{(0,0)}=
\det\left(\hess(u)|_{\textbf{0}}\right),
\end{equation}
and by the non-degeneracy assumption is non-zero. Thus, after the change of variables, the integral \eqref{eqn:int_norm} becomes
\begin{equation}
\int_V\frac{1}{W}f(x,y)dxdy=\int_{\varphi(V)}\frac{1}{|\det(\mathcal{J}\varphi)|\sqrt{\tilde x^2+\tilde y^2}}\tilde f(\tilde x,\tilde y)d\tilde xd\tilde y<+\infty.
\end{equation}
The claim \eqref{eqn:norm_nondeg} follows. Using estimate \eqref{eqn:mean_est}, we obtain the local integrability of $\mean$.
\end{proof}

\begin{rmk}
A non-degenerate characteristic point need not to be isolated, however, this situation is quite pathological. For example, in the Heisenberg group $\hei$ (cf.\ Section \ref{sec:hei}), the only situation in which this can occur is when we have a sequence of characteristic points $(x_n,y_n)$ accumulating at the origin, and not contained in any absolutely continuous curve. Indeed, consider $\surf=\{z=g(x,y)\}$ in $\hei$ and assume we have an absolutely continuous curve $\gamma\colon (-\varepsilon,\varepsilon)\rightarrow\surf$ of characteristic points with $\gamma(0)=\textbf{0}$, and such that $\dot\gamma(0)$ exists. Then, the origin is a degenerate characteristic point, indeed for each $t\in (-\varepsilon,\varepsilon)$, we have
\begin{equation}
\begin{cases}
g_x(\gamma_1(t),\gamma_2(t))+\frac{\gamma_2(t)}{2}=0,\\
g_y(\gamma_1(t),\gamma_2(t))-\frac{\gamma_1(t)}{2}=0.
\end{cases}
\end{equation} 
Differentiating both equations with respect to $t$, and evaluating at $t=0$, we have 
\begin{equation}
\begin{cases}
\dot\gamma_1(0)g_{xx}(0,0)+\dot\gamma_2(0)g_{xy}(0,0)+\frac{\dot\gamma_2(0)}{2}=0,\\
\dot\gamma_1(0)g_{xy}(0,0)+\dot\gamma_2(0)g_{yy}(0,0)-\frac{\dot\gamma_1(0)}{2}=0,
\end{cases}
\end{equation}
thus, $\left(\begin{smallmatrix}\dot\gamma_1(0)\\ \dot\gamma_2(0)\end{smallmatrix}\right)\in\ker(\hess(u)|_{\textbf{0}})$, implying that $\textbf{0}$ is a degenerate characteristic point.
\end{rmk}

\section{Integrability for mildly degenerate characteristic points in \texorpdfstring{$\hei$}{H}}\label{sec:hei}
The \emph{Heisenberg group} is the 3D contact structure on $\R^3$, defined by the 1-form
\begin{equation}
\omega=dz-\frac{1}{2}(xdy-ydx).
\end{equation}
A global frame for the contact distribution is given by $\{X,Y\}$, where
\begin{equation}
\label{eqn:hei_vf}
X=\partial_x-\frac{y}{2}\partial_z,\qquad Y=\partial_y+\frac{x}{2}\partial_z.
\end{equation}
Setting $\{X,Y\}$ to be an orthonormal frame, the resulting \sr manifold is the well-known first Heisenberg group, $\hei$. We equip it with the Lebesgue measure.

Let us consider in $\hei$ a surface $\surf=\{u=0\}$, where $u\in C^\infty(\R^3)$ with $du\neq 0$ on $\surf$. Assume that $p\in C(\surf)$ is a degenerate characteristic point, meaning that the horizontal Hessian of $u$ has zero determinant at $p$. Notice that, in $\hei$, the horizontal Hessian at $p\in C(\Sigma)$ coincides with the one introduced in \cite{MR2014879}, and, in terms of the orthonormal basis $\{X,Y\}$, takes the form:
\begin{equation}
\hess(u)|_p = \begin{pmatrix}
XXu(p) & XYu(p)\\
YXu(p) & YYu(p)
\end{pmatrix}.
\end{equation}
By the bracket-generating assumption, one of the entries of the Hessian must be non-zero at $p$, thus, it has a 1-dimensional kernel at $p$, spanned by some unitary vector, say $N_p\in\D_p$, which is unique, up to a sign. We extend $N_p$ to a left-invariant vector field $N\in\Gamma(\D)$. Taking an orthogonal vector field to $N$ in $\Gamma(\D)$, we obtain an orthonormal frame $\{N,\Nperp\}$ for the distribution, which, up to changing sign, we assume to be co-oriented with the standard one \eqref{eqn:hei_vf}.

\begin{defn}
\label{defn:crit_curve}
Let $\surf\subset \hei$ be a smooth embedded surface and let $p\in C(\surf)$ be degenerate. The critical curve of $p$ is defined as the set of points in $\surf$ where $N$ is tangent to $\surf$, i.e.
\begin{equation}
\label{eqn:crit_curve}
\mathcal{C}=\{q\in\surf \mid  N(q)\in T_q\surf\}.
\end{equation}
\end{defn}
We prove now that Definition \ref{defn:crit_curve} is well-posed.
\begin{lem}
\label{lem:smooth_curve}
Let $\surf\subset \hei$ be a smooth embedded surface and let $p\in C(\surf)$ be degenerate. Then, in a neighborhood of $p$, the set $\mathcal{C}$ as in \eqref{eqn:crit_curve}, is a smooth curve in $\surf$, trough $p$.
\end{lem}

\begin{proof}
Consider for $\surf$ a local defining function $u\in C^\infty$. Then, $N(q)\in T_q\surf$ if and only if $d_qu(N)=0$, thus
\begin{equation}
\mathcal{C}=\{u=0\}\cap\{Nu=0\}.
\end{equation}
In the orthonormal frame $\{N,\Nperp\}$, $Nu(p)=\Nperp u(p)=0$ since $p\in C(\surf)$.
However, $d_p u\neq 0$, therefore, by the bracket-generating assumption
\begin{equation}
[N,\Nperp]u(p)=N\Nperp u(p)-\Nperp Nu (p)\neq 0.
\end{equation}
We are going to show that $N\Nperp u(p)=0$. First of all, since the frame $\{N,\Nperp\}$ is co-oriented with $\{X,Y\}$, there exists $R\in \mathrm{SO}(2)$, such that 
\begin{equation}
\label{eqn:change_frame}
\begin{pmatrix}
N\\ \Nperp
\end{pmatrix}=\begin{pmatrix}
a&b\\-b&a
\end{pmatrix}\begin{pmatrix}
X\\Y
\end{pmatrix}
\end{equation}
where we used the shorthand $a=\cos(\theta)$, $b=\sin(\theta)$, for some $\theta\in [0,2\pi)$. Hence,
\begin{equation}
\begin{split}
N\Nperp u(p) &=(aX+bY)(-bX+aY)u(p)\\
						 &=-abXXu(p)+a^2XYu(p)-b^2YXu(p)+abYYu(p).
\end{split}
\end{equation}
Second of all, by definition of $N$, $N_p\in\ker(\hess(u)|_p)$, thus we obtain
\begin{equation}
0=\begin{pmatrix}
-b&a
\end{pmatrix}\begin{pmatrix}
XXu(p)&YXu(p)\\XYu(p)&YYu(p)
\end{pmatrix}\begin{pmatrix}
a\\b
\end{pmatrix}=N\Nperp u(p).
\end{equation}
Finally, $\Nperp Nu(p)\neq 0$ so the differential of the map $\phi=(u,Nu)$ has maximal rank at $p$, implying that $\mathcal{C}$ is a smooth curve in $\surf$, in a neighborhood of $p$. 
\end{proof}

\begin{rmk}
Notice that, in general, the critical curve $\mathcal{C}$ is not necessarily horizontal, and it is not related to the characteristic foliation induced on $\surf$ by the contact structure.
\end{rmk}

\begin{defn}\label{d:mildly}
Let $\surf\subset\hei$ be an embedded smooth surface, let $p\in C(\surf)$ be a degenerate characteristic point and let $u$ be a local defining function for $\Sigma$ around $p$. Let $\gamma:(-\varepsilon,\varepsilon) \to \mathcal{C}$ be a regular parametrization of $\mathcal{C}$, with $\gamma(0)=p$. We say that $p$ is mildly degenerate if the function 
\begin{equation}
s \mapsto Tu(\gamma(s)),
\end{equation}
has a finite order zero at $s=0$. Notice that this definition does not depend on the choice of $u$ and the regular parametrization.
\end{defn}

\begin{thm}\label{thm:int_deg}
Let $\surf\subset \hei$ be an embedded surface, let $p$ be an isolated mildly degenerate characteristic point and let $u\in C^\infty$ be a local defining function for $\surf$ in a neighborhood of $p$. Denote with $W$ the norm of the horizontal gradient of $u$. Then
\begin{equation}
\frac{1}{W}\in L^1_{\mathrm{loc}}(\surf,\sigma_\mathrm{R}),
\end{equation}
where $\sigma_\mathrm{R}$ denotes the Riemannian induced measure on $\surf$. In particular
\begin{equation}
\mean\in L^1_{\mathrm{loc}}(\surf,\sigma_\mathrm{R}). 
\end{equation}
\end{thm}

\begin{proof}
Without loss of generality, we may assume that $u(x,y,z)=z-g(x,y)$, i.e. $\surf=\{z=g(x,y)\}$, where $g\colon\R^2\rightarrow\R$ is smooth, and that $C(\surf)=\{\textbf{0}\}$. This implies
\begin{equation}
g(0,0)=\partial_xg(0,0)=\partial_yg(0,0)=0.
\end{equation}
Notice that the local integrability of $W$ is preserved by the action of isometries of $\hei$. We can exploit this fact to reduce $g$ to a normal form. Recall first that Heisenberg isometries (preserving the origin and the orientation of the z-axis) are given by the standard action of $\mathrm{SO}(2)$ on the $xy$-plane. Consider then the isometry 
\begin{equation}
\begin{pmatrix}
x\\y\\z
\end{pmatrix}\mapsto
\begin{pmatrix}
 \tilde x\\ \tilde y\\ \tilde z
\end{pmatrix}=
\begin{pNiceArray}{cc|c}[margin]
\Block{2-2}{R} & &  \\
& \hspace*{0.5cm} & \\
\hline
& &1
\end{pNiceArray}
\begin{pmatrix}
x\\y\\z
\end{pmatrix},
\end{equation}
where $R\in\mathrm{SO}(2)$ is defined in \eqref{eqn:change_frame}, in this way the frame $\{N,\Nperp\}$ is sent to the standard frame $\{X,Y\}$.  In particular, since $N_{\textbf{0}}\in \ker(\hess(u)|_{\textbf{0}})$, $N$ is given by
\begin{equation}
N=aX+bY=\frac{-YXu(p)X+XXu(p)Y}{\sqrt{YXu(p)^2+XXu(p)^2}}=\frac{\left(g_{11}+\frac{1}{2}\right)X-g_{20}Y}{\sqrt{\left(g_{11}+\frac{1}{2}\right)^2+g_{20}^2}},
\end{equation}
where $g_{ij}=\partial_x^i\partial_y^jg(0,0)$. Hence, the change of coordinates is given by
\begin{equation}
\begin{pmatrix}
 x\\y\\z
\end{pmatrix}=\frac{1}{W_0}
\begin{pmatrix}
 g_{11}+\frac{1}{2}&g_{20} &  \\
-g_{20}& g_{11}+\frac{1}{2} & \\
& &1
\end{pmatrix}
\begin{pmatrix}
 \tilde x\\ \tilde y\\ \tilde z
\end{pmatrix}=\frac{1}{W_0}\begin{pmatrix}
 \left(g_{11}+\frac{1}{2}\right)\tilde x+g_{20}\tilde y   \\
-g_{20}\tilde x +\left(g_{11}+\frac{1}{2}\right)\tilde y \\
\tilde z
\end{pmatrix},
\end{equation}
where $W_0=\sqrt{\left(g_{11}+\frac{1}{2}\right)^2+g_{20}^2}$.
Therefore, expanding $u$, we have
\begin{align}
u(x,y,z)&=z-\frac{g_{20}}{2}x^2-\frac{g_{02}}{2}y^2-g_{11}xy+O(r^3)\\
				&=\tilde z-\frac{g_{20}}{2W_0^2}\left(\left(g_{11}+\frac{1}{2}\right)\tilde x+g_{20}\tilde y\right)^2-\frac{g_{02}}{2W_0^2}\left(-g_{20}\tilde x +\left(g_{11}+\frac{1}{2}\right)\tilde y\right)^2\\
&\quad -\frac{g_{11}}{W_0^2}\left(\left(g_{11}+\frac{1}{2}\right)\tilde x+g_{20}\tilde y\right)\left(-g_{20}\tilde x +\left(g_{11}+\frac{1}{2}\right)\tilde y\right)+O(r^3),
\end{align}
where $r=\sqrt{x^2+y^2}$, and the coefficients of the second-order terms become
\begin{align}
{\tilde {x}}^2&\rightsquigarrow \frac{1}{W_0^2}\left(-\frac{g_{20}}{2}\left(g_{11}+\frac{1}{2}\right)^2+g_{11}\left(g_{11}+\frac{1}{2}\right)g_{20}-\frac{g_{02}}{2}g_{20}^2\right)=0,\\
\tilde x \tilde y&\rightsquigarrow \frac{1}{W_0^2}\left(-g_{20}\left(g_{11}+\frac{1}{2}\right)g_{20}-g_{11}\left(\left(g_{11}+\frac{1}{2}\right)^2-g_{20}^2\right)+g_{02}g_{20}\left(g_{11}+\frac{1}{2}\right)\right)=-\frac{1}{2},
\end{align}
having used the fact that the characteristic point is degenerate, which gives the condition $g_{20}g_{02}=g_{11}^2-\frac{1}{4}$. Hence, the function $g$ simplifies to
\begin{equation}
\label{eqn:norm_form}
g(\tilde x,\tilde y)=\frac{1}{2}\tilde x\tilde y+\frac{\alpha}{2}\tilde y^2+h(\tilde x,\tilde y),
\end{equation}
where $\alpha\in\R$ and $h\in C^\infty(\R^2)$ with order $\geq 3$. Notice that the specific value of $\alpha$ won't play any role in the integrability of $\mean$.

We can then assume that $u(x,y,z)=z-g(x,y)$ where $g$ has the normal form \eqref{eqn:norm_form}, and that $N=X$, $\Nperp=Y$. For such a function $u$, the norm of the horizontal gradient is 
\begin{equation}
W^2=(\alpha y+h_y)^2+(y+h_x)^2,
\end{equation}
so, since in these coordinates $\riemarea=f(x,y)dxdy$, where $f$ is a strictly positive and smooth function, we focus on 
\begin{equation}
\label{eqn:int_mean1}
\int_V\frac{1}{W}f(x,y)dxdy,
\end{equation}
where $V$ is a neighborhood of $(0,0)$. We may set $f\equiv 1$, since its explicit expression plays no role in the integrability. From Lemma \ref{lem:smooth_curve}, $\mathcal{C}=\{Nu=0\}\cap\surf$ is a smooth curve, whose expression in coordinates is $\{y+h_x=0\}\cap\surf$. Thus, we introduce the following smooth change of variables around the origin, rectifying $\mathcal{C}$ 
\begin{equation}
\label{eqn:change_var}
\varphi\colon
\begin{pmatrix}
x\\y
\end{pmatrix}
\mapsto
\begin{pmatrix}
x\\t
\end{pmatrix}
\quad\text{such that}\quad
\begin{cases}
x=x\\
t=y+h_x(x,y)
\end{cases}
\end{equation} 
and the integral \eqref{eqn:int_mean1} becomes
\begin{equation}
\label{eqn:int_mean2}
\int_{V'}\frac{1}{\left((\alpha t+(h_y-\alpha h_x))^2+t^2\right)^{1/2}|1+2h_{xy}|}dtdx,
\end{equation}
where the integrand is evaluated in $\varphi^{-1}(x,t)$. Here $V'=\varphi(V)$. We expand in Taylor series the function $h_y-\alpha h_x$, with respect to the $t$-variable at the point $(x,0)$, obtaining
\begin{equation}
\label{eqn:xi_defn}
h_y(\varphi^{-1}(x,t))-\alpha  h_x(\varphi^{-1}(x,t))=\xi(x)+tR(x,t),
\end{equation}
where $\xi,R$ are smooth functions of order $\geq 2$ and $\geq 1$ respectively, since $h$ was of order at least 3 and the notion of order in $t$ is preserved by $\varphi$. But now, parametrizing the critical curve by $x\mapsto \gamma(x) = (x,y(x),g(x,y(x)))$ where $y(x) +h_x(x,y(x))=0$, we have that $\varphi^{-1}(x,0)=(x,y(x))$ and
\begin{equation}
\xi(x)=-\alpha h_x(\varphi^{-1}(x,0))+h_y(\varphi^{-1}(x,0))=-\alpha h_x(x,y(x))+h_y(x,y(x))=Tu(\gamma(x)).
\end{equation}
Thus, by assumption of mildly degenerate characteristic point, $\xi$ has a zero of finite order at $x=0$. So, we may write
\begin{equation}
\xi(x)=c_0x^k(1+r(x)),
\end{equation}
where $k$ is an integer $\geq 2$, and $r$ is a smooth function of order $\geq 1$. Thus, we introduce the following weighted polar coordinates in the plane
\begin{equation}
\label{eqn:polar}
\psi\colon
\begin{pmatrix}
x\\t
\end{pmatrix}
\mapsto
\begin{pmatrix}
\rho\\ \theta
\end{pmatrix}
\quad\text{such that}\quad
\begin{cases}
c_0x^k=\rho\cos(\theta)\\
(\alpha^2+1)^{1/2}t=\rho\sin(\theta)
\end{cases}
\end{equation} 
whose Jacobian is $\frac{1}{(\alpha^2+1)^{1/2}kc_0}\rho^{1/k}|\cos(\theta)|^{1/k-1}$. In these new coordinates, the function $W$ becomes
\begin{equation}
\begin{split}
W^2&=(\alpha t+\xi(x)+tR(x,t))^2+t^2=\rho^2\left(1+\frac{\alpha\sin(2\theta)}{(1+\alpha^2)^{1/2}}+R_{\mathrm{pol}}(\rho,\theta)\right),
\end{split}
\end{equation}
where $R_{\mathrm{pol}}(\rho,\theta)$ is a remainder term vanishing at $\rho=0$. Therefore, the integral  \eqref{eqn:int_mean2} is controlled by
\begin{equation}
\int_{V''}\frac{\left|\rho\cos(\theta)\right|^{1/k-1}}{\sqrt{1+\frac{\alpha\sin(2\theta)}{(1+\alpha^2)^{1/2}}+R_{\mathrm{pol}}(\rho,\theta)}}d\rho d\theta,
\end{equation}
where $V''=\psi(\varphi(V))$. But now this integral is finite, since 
\begin{equation}
1+\frac{\alpha\sin (2\theta)}{(\alpha^2+1)^{1/2}}>1-\frac{|\alpha|}{(\alpha^2+1)^{1/2}}>0
\end{equation}
and thus the denominator, up to restricting $V''$, is never-vanishing.
\end{proof}

\begin{rmk}
\label{rmk:sharpness}
The mild degeneration assumption is sharp for the local integrability of $W^{-1}$. Consider the example taken from \cite[Prop. 3.4]{DGN-Integrability} where $\surf=\{z=g(x,y)\}$, with
\begin{equation}
g(x,y)=\frac{1}{2}xy+\frac{1}{2}y^2+\int_0^xe^{-\tau^{-2}}d\tau.
\end{equation}
Here $N=X$ and $\Nperp=Y$, being $g$ in the normal form \eqref{eqn:norm_form}. Then, the critical curve of $\textbf{0}$ is $\mathcal{C}=\{y+e^{-x^{-2}}=0\}\cap\surf$, which can be parametrized by
\begin{equation}
\gamma(x)=\left(x,-e^{-x^{-2}},g\left(x,e^{-x^{-2}}\right)\right).
\end{equation}
Thus, $Tu(\gamma(x))=-e^{-x^{-2}}$, which has infinite order at $x=0$. Therefore, $\textbf{0}$ is not a mildly degenerate characteristic point and one can check that $W^{-1}$ is not locally integrable.

Notice, however, that in the previous example, $\mean$ is locally integrable. Thus, in general, to prove the integrability of $\mean$, one should take into account also its numerator, which vanishes at characteristic points.
\end{rmk}

\begin{thm}\label{thm:realanal}
Let $\surf\subset \hei$ be a real-analytic embedded surface, let $p$ be an isolated characteristic point and let $u\in C^\omega$ be a local defining function for $\surf$ in a neighborhood of $p$. Denote with $W$ the norm of the horizontal gradient of $u$. Then
\begin{equation}
\frac{1}{W}\in L^1_{\mathrm{loc}}(\surf,\sigma_\mathrm{R}),
\end{equation}
where $\sigma_\mathrm{R}$ denotes the Riemannian induced measure on $\surf$. In particular
\begin{equation}
\mean\in L^1_{\mathrm{loc}}(\surf,\sigma_\mathrm{R}). 
\end{equation}
\end{thm}

\begin{proof}
If $C(\surf)$ consists of non-degenerate characteristic points, the result follows from Theorem \ref{thm:int_nondeg}. If $p\in C(\surf)$ is degenerate, we show that $p$ is actually mildly degenerate. We can assume that $\surf=\{z-g(x,y)=0\}$, where $g\in C^\omega(\R^2)$ has the normal form \eqref{eqn:norm_form}, and $\textbf{0}\in C(\surf)$ is degenerate. In this case, in coordinates $(x,t)=\varphi(x,y)$ defined in \eqref{eqn:change_var}, the critical curve of $\textbf{0}$ is $\mathcal{C}=\{t=0\}$ and $p\in C(\surf)$ if and only if
\begin{equation}
\begin{cases}
\xi(x)=0,\\
t=0.
\end{cases}
\end{equation}
Since $\textbf{0}$ is an isolated characteristic point, $\xi$ is not identically zero. Thus, since $\xi$ is real-analytic it has finite order at $x=0$.
\end{proof}

\bibliographystyle{alphaabbr}
\bibliography{biblio-mean}

\end{document}